\documentclass[a4paper,11pt,openany,reqno]{article}

\usepackage[a4paper,tmargin= 4.5cm,bmargin= 4.5cm,rmargin=3.5cm,lmargin=3.5cm]{geometry}

\usepackage{calc}
\usepackage[all]{xy}
\usepackage[centertags]{amsmath}
\usepackage{latexsym}
\usepackage{amsfonts}
\usepackage{amssymb}
\usepackage{amsthm}
\usepackage{fancyhdr}
\usepackage [dvips]{epsfig}
\usepackage{newlfont}
\usepackage[latin1]{inputenc}
\usepackage[french,english]{babel}
\usepackage{graphicx}
\usepackage{t1enc}
\usepackage {color}
\usepackage{multicol}
\usepackage{mathpazo}

\newcommand{\dpq}{d_{p,q}}
\newcommand{\Dpq}{D_{p,q}}

\newcommand{\pqbinomial}[4]{\mbox{$
\biggl[\!\! 
\begin{array}{c}
#1\\
 #2
\end{array}\!\!\biggr]_{
\!{#3,#4}} $} }
\newcommand{\qbinomial}[3]{\mbox{$
\biggl[ 
\begin{array}{c}
#1\\
 #2
\end{array}\biggr]_{
\!{#3}} $} }

\newtheorem{remark}{Remark}
\theoremstyle{plain}
\newtheorem{lemma}{Lemma}
\newtheorem{proposition}{Proposition}
\newtheorem{theorem}{Theorem}
\newtheorem{definition}{{Definition}}
\newtheorem{corollary}{Corollary}

\makeatletter 
\def \be {\begin{equation}}
\def \ee {\end{equation}}

\def \btb {\begin{tablen}}
\def \etb {\end{tablen}}

\def \bea {\begin{eqnarray}}
\def \eea {\end{eqnarray}}
\def \bean {\begin{eqnarray*}}
\def \eean {\end{eqnarray*}}
\def \bd {\begin{definition}}
\def \ed {\end{definition}}
\def \bl {\begin{lemma}}
\def \el {\end{lemma}}
\def \bcon {\begin{conjecture}}
\def \econ {\end{conjecture}}
\def \bp {\begin{proposition}}
\def \ep {\end{proposition}}
\def  \br {\begin{remark}}
\def \er {\end{remark}}
\def \bt {\begin{theorem}}
\def \et {\end{theorem}}
\def \bc {\begin{corollary}}
\def \ec {\end{corollary}}
\def \been {\begin{enumerate}}
\def \een {\end{enumerate}}
\def \bit {\begin{itemize}}
\def \eit {\end{itemize}}

\title{On the fundamental theorem of $(p,q)$-calculus \\ and some $(p,q)$-Taylor formulas}
\author{{\textbf{P. Njionou Sadjang\footnote{Email adress: \texttt{pnjionou@yahoo.fr}}} }}

\begin{document}

\maketitle

\noindent 
\begin{abstract}
\noindent
In this paper, the $(p,q)$-derivative and the $(p,q)$-integration are investigated. Two suitable polynomials bases for the $(p,q)$-derivative are provided and various properties of these bases are given. As application, two $(p,q)$-Taylor formulas for polynomials are given, the fundamental theorem of $(p,q)$-calculus is included and the formula of $(p,q)$-integration by part is proved.\\
\hrulefill\\
\noindent{\it Keywords:}  $(p,q)$-derivative, $(p,q)$-integration, $(p,q)$-Taylor formula, fundamental theorem, $(p,q)$-integration by part.\\

\noindent {\bf AMS Subject Classification (2010)}: 33D15, 33D25, 33D35.

\end{abstract}



\section{Introduction}

\noindent The Taylor formula for polynomials $f(x)$ evaluates the coefficients $f_k$ in the expansion
\begin{equation}\label{tay1}
f(x)=\sum_{k=0}^{\infty}f_k(x-c)^k,\quad f_k=\dfrac{f^{(k)}(c)}{k!}.
\end{equation} 
It is possible to generalize (\ref{tay1}) by considering other polynomial bases and suitable operators.

\noindent The fundamental theorem of calculus can be stated as follows.
\begin{theorem}
If $f$ is a continuous function on an interval $(a;b)$, then $f$ has an antiderivative on $(a;b)$. Moreover, if $F$ is any antiderivative of $f$ on $(a;b)$, then 
\begin{equation}
  \int_{a}^{b}f(x)dx=F(b)-F(a).
\end{equation}
\end{theorem}
\noindent The $q$ version of this theorem was stated in \cite{kac} as follows.
\begin{theorem} 
If $F(x)$ is an antiderivative of $f(x)$ and if $F(x)$ is continuous at $x=0$, then
\begin{equation}\label{qft}
 \int_a^b f(x)d_qx=F(b)-F(a),\quad 0\leq a\leq b\leq\infty.
\end{equation}
\end{theorem}

\noindent Here the $q$-integral is defined by 
\begin{equation}
\int_0^a f(x)d_qx=(1-q)a\sum_{k=0}^{\infty}q^kf(aq^k).
\end{equation}

\noindent In this paper, two generalizations of (\ref{tay1}) are given and a generalization of (\ref{qft}) is stated.
The paper is organised as follows.
\begin{itemize}
  \item In Section {\bf 2}, we introduce and give relevant properties of the $(p,q)$-derivative. The $(p,q)$-power basis is given and main of its properties are provided. The properties of the $(p,q)$-derivative combined with those of the $(p,q)$-power basis enable to state two $(p,q)$-Taylors for polynomials. It then follows connection formulas between the canonical basis and the $(p,q)$-power basis.   
  \item In Section {\bf 3}, the $(p,q)$-antiderivative, the $(p,q)$-integral are introduced and sufficient condition for their convergence are investigated. Finally the fundamental theorem of $(p,q)$-calculus is proved and the formula of $(p,q)$-integration by part is derived.
\end{itemize}

%


\section{The $(p,q)$-derivative and the $(p,q)$-power basis}

\noindent In this section, we introduce the $(p,q)$-derivative, the $(p,q)$-power and provide some of their relevant properties. Two $(p,q)$-Taylor formulas for polynomials are stated and some consequences are investigated.

\subsection{The $(p,q)$-derivative}

\noindent Let $f$ be a function defined on the set of the complex numbers.
\begin{definition}
The $(p,q)$-derivative of  the function $f$  is defined as  (see e.g. \cite{Jagannathan2006,desire})   
\begin{equation}
\Dpq f(x)=\dfrac{f(px)-f(qx)}{(p-q)x},\quad x\neq0,
\end{equation}
and $(\Dpq f)(0)=f'(0)$,
provided that $f$ is differentiable at $0$.
 The so-called $(p,q)$-bracket or  twin-basic number is defined as
\begin{equation}\label{pqnumber}
[n]_{p,q}=\frac{p^n-q^n}{p-q}.
\end{equation}
\end{definition}

\noindent It happens clearly that  $\Dpq x^n=[n]_{p,q}x^{n-1}$. Note also that for $p=1$, the $(p,q)$-derivative reduces to the Hahn derivative given by
\[D_qf(x)=\dfrac{f(x)-f(qx)}{(1-q)x},\quad x\neq 0.\]

\noindent As with ordinary derivative, the action of the $(p,q)$-derivative of a function is a linear operator. More precisely, for any constants $a$ and $b$,
\begin{eqnarray*}
\Dpq (af(x)+bg(x))&=&a\Dpq f(x)+b\Dpq g(x).
\end{eqnarray*}

\noindent The twin-basic number is a natural generalization of the $q$-number, that is
\begin{equation}
\lim\limits_{p\to 1}[n]_{p,q}=[n]_q=\frac{1-q^n}{1-q},\quad q\neq 1.
\end{equation}
\noindent The $(p,q)$-factorial is defined by 
\begin{equation}
[n]_{p,q}!=\prod_{k=1}^{n}[k]_{p,q}!,\quad n\geq 1,\quad [0]_{p,q}!=1.
\end{equation}
\noindent Let us introduce also the so-called $(p,q)$-binomial coefficient
\begin{equation}\label{pqbin}
\pqbinomial{n}{k}{p}{q}=\dfrac{[n]_{p,q}!}{[k]_{p,q}![n-k]_{p,q}!}, \quad 0\leq k\leq n.
\end{equation}
are called $(p,q)$-binomial coefficients. Note that as $p\to 1$, the $(p,q)$-binomial coefficients reduce to the $q$-binomial coefficients.

\begin{proposition}
The $(p,q)$-derivative fulfils the following product rules
\begin{eqnarray}
\Dpq (f(x)g(x))&=& f(px)\Dpq g(x)+g(qx)\Dpq f(x),\label{productrule2}\\
\Dpq (f(x)g(x))&=& g(px)\Dpq f(x)+f(qx)\Dpq g(x)\label{productrule3}
\end{eqnarray}
\end{proposition}

\begin{proof}
From the definition of the $(p,q)$-derivative, we have
\begin{eqnarray*}
\Dpq (f(x)g(x))&=&\frac{f(px)g(px)-f(qx)g(qx)}{(p-q)x}\\
&=&\frac{f(px)[g(px)-g(qx)]+g(qx) [f(px)-f(qx)]}{(p-q)x}\\
&=&f(px)\Dpq g(x)+g(qx)\Dpq f(x).
\end{eqnarray*}
This proves (\ref{productrule2}). (\ref{productrule3}) is obtained by symmetry. 
\end{proof}

\begin{proposition}
The $(p,q)$-derivative fulfils the following product rules
\begin{eqnarray}
\Dpq \left(\frac{f(x)}{g(x)}\right)&=&\dfrac{g(qx)\Dpq f(x)-f(qx)\Dpq g(x)}{g(px)g(qx)}\label{quotient1}\\
\Dpq \left(\frac{f(x)}{g(x)}\right)&=&\dfrac{g(px)\Dpq f(x)-f(px)\Dpq g(x)}{g(px)g(qx)}\label{quotient2}
\end{eqnarray}
\end{proposition}

\begin{proof}The proof of this statements can be deduced using (\ref{productrule2}).
\end{proof}

\subsection{The $(p,q)$-power basis}
\noindent Here, we introduce the so-called $(p,q)$-power and investigate some of its relevant properties. \\
The expression 
\begin{equation}
 (x\ominus a)_{p,q}^n=(x-a)(px-aq)\cdots (px^{n-1}-aq^{n-1})
\end{equation}
is called the $(p,q)$-power. These polynomials will be useful to state our Taylor formulas.

\begin{proposition}
The following assertion is valid.
\begin{equation}\label{derule1}
\Dpq (x\ominus a)_{p,q}^n=[n]_{p,q}(px\ominus a)_{p,q}^{n-1}, \quad n\geq 1,
\end{equation}
and $\Dpq (x\ominus a)_{p,q}^0=0$.
\end{proposition}

\begin{proof} The proof follows by a direct computation.
\end{proof}

\begin{proposition}
Let $\gamma$ be a complex number and $n\geq 1$ be an integer, then
\begin{equation}\label{derule2}
\Dpq (\gamma x\ominus a)_{p,q}^n=\gamma [n]_{p,q}(\gamma px\ominus a)_{p,q}^{n-1}.
\end{equation}
\end{proposition}

\begin{proof}
The proof is done exactly as the proof of  (\ref{derule1}).
\end{proof}
We now generalize (\ref{derule1}) in the following proposition.

\begin{proposition}
Let $n\geq 1$ be an integer, and $0\leq k\leq n$, the following rule applies
\begin{equation}\label{derule3}
\Dpq ^k(x\ominus a)_{p,q}^{n}=p^{\binom{k}{2}}\frac{[n]_{p,q}!}{[n-k]_{p,q}!}(p^kx\ominus a)_{p,q}^{n-k}.
\end{equation} 
\end{proposition}

\begin{proof}
The prove is done by induction with respect to $k$. 
\end{proof}

\begin{remark}
For the classical derivative, it is known that for any complex number $\alpha$, one has
\[\dfrac{d}{dx}x^{\alpha}=\alpha x^{\alpha-1}.\]
In what follows, we would like to state similar result for the $\Dpq$ derivative as done for the $D_q$ derivative in \cite{kac}. 
\end{remark}

\begin{proposition}\label{proexpand}
Let $m$ and $n$ be two non negative integers. Then the following assertion is valid.
\begin{equation}\label{expand1}
(x\ominus a)_{p,q}^{m+n}=(x\ominus a)_{p,q}^{m}(p^mx\ominus q^ma)_{p,q}^{n}.
\end{equation}
\end{proposition}

\noindent In Proposition \ref{proexpand}, if we take $m=-n$, then we get the following extension of the $(p,q)$-power basis.
\begin{definition} Let $n$ be a non negative integer, then we set the following definition.
\begin{equation}\label{negdef}
(x\ominus a)_{p,q}^{-n}=\frac{1}{(p^{-n}x\ominus q^{-n}a)_{p,q}^{n}}.
\end{equation}
\end{definition}

\begin{proposition}
For any two integers $m$ and $n$, (\ref{expand1}) holds.
\end{proposition}

\begin{proof}
The case $m>0$ and $n>0$ has already been proved, and the case where one of $m$ and $n$ is zero is easy.  Let us first consider the case\linebreak $m=-m'<0$ and $n>0$. Then,
\begin{eqnarray*}
(x\ominus a)_{p,q}^{m}(p^mx\ominus q^ma)_{p,q}^{n}&=&(x\ominus a)_{p,q}^{-m'}(p^{-m'}x\ominus q^{-m'}a)_{p,q}^{n}\\
\textrm{{\small by (\ref{negdef})}}\quad&=&\frac{(p^{-m'}x\ominus q^{-m'}a)_{p,q}^{n}}{(p^{-m'}x\ominus q^{-m'}a)_{p,q}^{m'}}\\
\textrm{{\small by (\ref{expand1})}}\quad&=&\left\{\begin{array}{ll}
\left(p^m(p^{-m}x)\ominus q^m(q^{-m}a)\right)_{p,q}^{n-m'}& \textrm{if }\;\; n\geq m'\\
\frac{1}{\left(q^n(q^{-m'}x)\ominus q^n(q^{-m'}a)\right)_{p,q}^{m'-n}}& \textrm{if }\;\;n<m'
\end{array}\right.\\
\textrm{{\small by (\ref{negdef})}}\quad&=&(x\ominus a)_{p,q}^{n-m'}=(x\ominus a)^{n+m}_{p,q}.
\end{eqnarray*}
If $m\geq 0$ and $n=-n'<0$, then
\begin{eqnarray*}
(x\ominus a)_{p,q}^{m}(p^mx\ominus q^ma)_{p,q}^{n}&=&(x\ominus a)_{p,q}^{m}(p^{m}x\ominus q^{m}a)_{p,q}^{-n'}\\
&=&\frac{(x\ominus a)_{p,q}^{m}}{(p^{m-n'}x\ominus q^{m-n'}a)_{p,q}^{n'}}\\
&=&\left\{ \begin{array}{ll}
\frac{(x\ominus a)_{p,q}^{m-n'}(p^{m-n'}x\ominus aq^{m-n'})_{p,q}^{n'}}{(p^{m-n'}x\ominus q^{m-n'}a)_{p,q}^{n'}}&\textrm{if}\;\; m>n'\\
\frac{(x\ominus a)_{p,q}^{m}}{(p^{m-n'}x\ominus q^{m-n'}a)_{p,q}^{n'-m}(p^{n'-m}(p^{m-n'})x\ominus q^{n'-m} (q^{m-n'}a))_{p,q}^{m}}& \textrm{if}\;\; m<n'
\end{array}\right.\\
&=&\left\{ \begin{array}{ll}
(x\ominus a)_{p,q}^{m-n'}&\textrm{if}\;\; m>n'\\
\frac{1}{(p^{m-n'}x\ominus q^{m-n'}a)_{p,q}^{n'-m}}& \textrm{if}\;\; m<n'
\end{array}\right.\\
&=&(x\ominus a)_{p,q}^{m-n'}=(x\ominus a)_{p,q}^{m+n}.
\end{eqnarray*}
Lastly, if $m=-m'<0$ and $n=-n'<0$, 
\begin{eqnarray*}
(x\ominus a)_{p,q}^{m}(p^mx\ominus q^ma)_{p,q}^{n}&=&(x\ominus a)_{p,q}^{-m'}(p^{-m'}x\ominus q^{-m'}a)_{p,q}^{-n'}\\
&=&\frac{1}{(p^{-m'}x\ominus q^{-m'}a)_{p,q}^{m'}(p^{-n'-m'}x\ominus q^{-n'-m'}a)_{p,q}^{n'}}\\
&=& \frac{1}{(p^{-n'-m'}x\ominus q^{-n'-m'}a)_{p,q}^{n'+m'}}\\
&=&(x\ominus a)^{-m'-n'}_{p,q}=(x\ominus a)_{p,q}^{m+n}.
\end{eqnarray*}
Therefore, (\ref{expand1}) is true for any integers $m$ and $n$.
\end{proof}

\noindent It is natural to ask ourselves if (\ref{derule1}) is valid for any integer as well. But before trying to answer this question, let us generalise the twin-basic number as follows.

\begin{definition}
Let $\alpha$ be any number,
\begin{equation}
[\alpha]_{p,q}=\dfrac{p^\alpha-q^\alpha}{p-q}.
\end{equation}
\end{definition}

\begin{proposition}
For any integer $n$, 
\begin{equation}\label{der3}
\Dpq (x\ominus a)^n_{p,q}=[n]_{p,q}(px\ominus a)^{n-1}_{p,q}.
\end{equation}
\end{proposition}

\begin{proof}
Note that $[0]=0$. The result is already proved for $n\geq 0$. For  $n=-n'<0$, we use (\ref{quotient1}) and (\ref{negdef}) to get the result.
\end{proof}


\begin{proposition}
The following relations are valid:
\begin{eqnarray}
   \Dpq \dfrac{1}{(x\ominus a)_{p,q}^n}&=&\dfrac{-q[n]_{p,q}}{(qx\ominus a)_{p,q}^{n+1}},\label{r1}\\
   \Dpq (a\ominus x)_{p,q}^{n}&=& -[n]_{p,q}(a\ominus qx)_{p,q}^{n-1},\label{r2}\\
   \Dpq \dfrac{1}{(a\ominus x)_{p,q}}&=&  \dfrac{p[n]_{p,q}}{(a\ominus px)_{p,q}^{n+1}}.  \label{r3} 
\end{eqnarray}
\end{proposition}

\begin{proof} The proof follows by direct computations.
\end{proof}

\begin{proposition}
Let $n\geq 1$ be an integer, and $0\leq k\leq n$, we have the following
\begin{equation}\label{derule4}
\Dpq ^k(a\ominus x)_{p,q}^{n}=(-1)^kq^{\binom{k}{2}}\frac{[n]_{p,q}!}{[n-k]_{p,q}!}(a\ominus q^kx)_{p,q}^{n-k}.
\end{equation} 
\end{proposition}

\begin{proof}
The prove is done by induction with respect to $k$. 
\end{proof}

\subsection{$(p,q)$-Taylor formulas for polynomials}
In this section, two Taylors formulas for polynomials are given and some of their consequences are investigated.
\begin{theorem}
For any polynomial $f(x)$ of degree $N$, and any number $a$, we have the following $(p,q)$-Taylor expansion:
\begin{equation}\label{pqtaylor}
f(x)=\sum_{k=0}^{N}p^{-\binom{k}{2}}\frac{\left(\Dpq^k f\right)(ap^{-k})}{[k]_{p,q}!}(x\ominus a)_{p,q}^k.
\end{equation}
\end{theorem}

\begin{proof}
Let $f$ be a polynomial of degree $N$, then we have the expansion 
\begin{equation}\label{preuve1}
f(x)=\sum_{j=0}^Nc_j(x\ominus a)_{p,q}^j.
\end{equation}
Let $k$ be an integer such that $0\leq k\leq N$, then, applying $\Dpq^k$ on both sides of  (\ref{preuve1}) and using (\ref{derule3}), we get
\[\left(\Dpq^k f\right)(x)=\sum_{j=k}^Nc_j\frac{[j]_{p,q}!}{[j-k]_{p,q}!}p^{\binom{k}{2}}(p^kx\ominus q)_{p,q}^{j-k}.\]
Substituting $x=ap^{-k}$, it follows that 
\[\left(\Dpq^k f\right)(ap^{-k})=c_k[k]_{p,q}!p^{\binom{k}{2}},\]
thus we get 
\[c_k=p^{-\binom{k}{2}}\dfrac{\left(\Dpq^k f\right)(ap^{-k})}{[k]_{p,q}!}.\]
This proves the desired result.
\end{proof}

\begin{corollary}
The following connection formula holds.
\begin{equation}
x^n=\sum_{k=0}^np^{-\binom{k}{2}}\pqbinomial{n}{k}{p}{q}(ap^{-k})^{n-k}(x\ominus a)_{p,q}^k \label{conec1}
\end{equation}
\end{corollary}


\begin{theorem}
For any polynomial $f(x)$ of degree $N$, and any number $a$, we have the following $(p,q)$-Taylor expansion:
\begin{equation}\label{pqtaylor2}
f(x)=\sum_{k=0}^{N}(-1)^kq^{-\binom{k}{2}}\frac{\left(\Dpq^k f\right)(aq^{-k})}{[k]_{p,q}!}(a\ominus x)_{p,q}^k.
\end{equation}
\end{theorem}

\begin{proof}
Let $f$ be a polynomial of degree $N$, then we have the expansion 
\begin{equation}\label{preuve10}
f(x)=\sum_{j=0}^Nc_j(a\ominus x)_{p,q}^j.
\end{equation}
Let $k$ be an integer such that $0\leq k\leq N$, then, applying $\Dpq^k$ on both sides of  (\ref{preuve10}) and using (\ref{derule4}), we get
\[\left(\Dpq^k f\right)(x)=\sum_{j=k}^Nc_j(-1)^j\frac{[j]_{p,q}!}{[j-k]_{p,q}!}q^{-\binom{k}{2}}(a\ominus q^kx)_{p,q}^{j-k}.\]
Substituting $x=aq^{-k}$, it follows that 
\[\left(\Dpq^k f\right)(aq^{-k})=c_k(-1)^k[k]_{p,q}!q^{-\binom{k}{2}},\]
thus we get 
\[c_k=(-1)^kq^{-\binom{k}{2}}\dfrac{\left(\Dpq^k f\right)(aq^{-k})}{[k]_{p,q}!}.\]
This proves the desired result.
\end{proof}

\begin{corollary}
The following connection formula holds.
\begin{equation}
x^n=\sum_{k=0}^n(-1)^kq^{-\binom{k}{2}}\pqbinomial{n}{k}{p}{q}(aq^{-k})^{n-k}(a\ominus x)_{p,q}^k. \label{conec2}
\end{equation}
\end{corollary}

%

%

\begin{corollary}
The following connection formulas hold.
\begin{eqnarray}
(x\ominus b)_{p,q}^{n}&=&\sum_{k=0}^{n}\pqbinomial{n}{k}{p}{q}(a\ominus b)_{p,q}^{n-k}(x\ominus a)_{p,q}^{k},\label{conecc3}\\
(b\ominus x)_{p,q}^{n}&=&\sum_{k=0}^n\pqbinomial{n}{k}{p}{q}(b\ominus a)_{p,q}^{n-k}(a\ominus x)_{p,q}^{k},\label{conecc4}
\end{eqnarray}

\end{corollary}

\begin{remark}
If one takes $b=ab$ in (\ref{conecc3}), then one gets
\[(x\ominus ab)_{p,q}^{n}=\sum_{k=0}^{n}\pqbinomial{n}{k}{p}{q}a^{n-k}(1\ominus b)_{p,q}^{n-k}(x\ominus a)_{p,q}^{k}.\]
Now, take $x=1$ and $p=1$, the following well known $q$-binomial theorem follows 
\begin{equation}
(ab;q)_n=\sum_{k=0}^{n}\qbinomial{n}{k}{q}a^{n-k}(b;q)_{n-k}(a;q)_{k}.\label{qbin}
\end{equation}
Then, (\ref{conecc3}) is an obvious generalization of (\ref{qbin}).
\end{remark}

\begin{corollary}
The following expansion holds.
\begin{eqnarray}
\dfrac{1}{(1\ominus x)_{p,q}^n}&=&1+\sum_{j=0}^{\infty}\dfrac{p^{j-\binom{j}{2}}[n]_{p,q}[n+1]_{p,q}\cdots [n+j-1]_{p,q}}{[j]_{p,q}!}x^n\nonumber \\
&=& 1+\sum_{j=0}^{\infty}\pqbinomial{n+j-1}{j}{p}{q}p^{j-\binom{j}{2}}{x^j},\label{heine}
\end{eqnarray}
\end{corollary}

\noindent Note that (\ref{heine}) is the $(p,q)$-analogue of the Taylor's expansion of $f(x)=\dfrac{1}{(1-x)^n}$ in ordinary calculus.
Note also that when $p\to 1$, (\ref{heine}) becomes the well known Heine's binomial formula.

\section{The $(p,q)$-antiderivative and the $(p,q)$-integral}

\subsection{The $(p,q)$-antiderivative}

\noindent The function $F(x)$ is a $(p,q)$-antiderivative of $f(x)$ if  $\Dpq F(x)=f(x)$. It is denoted by
\begin{equation}
\int f(x)\dpq x.
\end{equation}

\noindent Note that we say "a" $(p,q)$-antiderivative instead of "the"  $(p,q)$-antiderivative, because, as in ordinary calculus, an antiderivative is not unique. In ordinary calculus, the uniqueness is up to a constant since the derivative of a function vanishes if and only if it is a constant. The situation in the twin basic quantum calculus is more subtle. $\Dpq\varphi(x)=0$ if and only if $\varphi(px)=\varphi(qx)$, which does not necessarily imply $\varphi$ a constant. If we require $\varphi$ to be a formal power series, the condition $\varphi(px)=\varphi(qx)$ implies $p^nc_n=q^nc_n$ for each $n$, where $c_n$ is the coefficient of $x^n$. It is possible only when $c_n=0$ for any $n\geq 1$, that is, $\varphi$ is constant. Therefore, if 
\[f(x)=\sum_{n=0}^{\infty}a_nx^n\]
is a formal power series, then among formal power series, $f(x)$ has a unique $(p,q)$-antiderivative up to a constant term, which is 
\begin{equation}
\int f(x)\dpq x=\sum_{n=0}^{\infty}\dfrac{a_n x^{n+1}}{[n+1]_{p,q}}+C.
\end{equation}

\subsection{The $(p,q)$-integral}

We define the inverse of the $(p,q)$-differentiation called the $(p,q)$-integration. Let $f(x)$ be an arbitrary function and $F(x)$ be a function such that $\Dpq F(x)=f(x)$, then
\[\dfrac{F(px)-F(qx)}{(p-q)x}=f(x).\]
Therefore, $F(px)-F(qx)= \varepsilon x f(x)$ where $\varepsilon=(p-q)$. This relation leads to the formula 
\begin{eqnarray*}
F\left(p^{1}q^{-1}x\right)-F\left(p^{0}q^{-0}x\right)&=& \varepsilon p^{0}q^{-1}x f\left(p^{0}q^{-1}x\right)\\
F\left(p^{2}q^{-2}x\right)-F\left(p^{1}q^{-1}x\right)&=& \varepsilon p^{1}q^{-2}x f\left(p^{1}q^{-2}x\right)\\
F\left(p^{3}q^{-3}x\right)-F\left(p^{2}q^{-2}x\right)&=& \varepsilon p^{2}q^{-3}x f\left(p^{2}q^{-3}x\right)\\
&\vdots& \\
F\left(p^{n+1}q^{-(n+1)}x\right)-F\left(p^{n}q^{-n}x\right)&=& \varepsilon p^{n}q^{-(n+1)}x f\left(p^{n}q^{-(n+1)}x\right)\\
\end{eqnarray*}
By adding these formulas terms by terms, we obtain 
\[F\left(p^{n+1}q^{-(n+1)}x\right)-F(x)=(p-q)x\sum_{k=0}^{n}f\left(p^{k}q^{-(k+1)}x\right).\]
Assuming $\left|\dfrac{p}{q}\right|<1$ and letting $n\to\infty$, we have 
\[F(x)-F(0)=(q-p)x\sum_{k=0}^{\infty}\frac{p^{k}}{q^{k+1}}f\left(\frac{p^{k}}{q^{k+1}}x\right).\]
Similarly, for $\left|\dfrac{p}{q}\right|>1$, we have 
\[F(x)-F(0)=(p-q)x\sum_{k=0}^{\infty}\frac{q^{k}}{p^{k+1}}f\left(\frac{q^{k}}{p^{k+1}}x\right).\]
Therefore, we give the following definition.

\begin{definition} Let $f$ be an arbitrary function. We define the $(p,q)$-integral of $f$ as follows:
\begin{equation}\label{pqantider}
\int f(x)\dpq x=(p-q)x\sum_{k=0}^{\infty}\frac{q^{k}}{p^{k+1}}f\left(\frac{q^{k}}{p^{k+1}}x\right).
\end{equation}
\end{definition}

\begin{remark}
Note that this is a formal definition since the we do not care about the convergence of the right hand side of (\ref{pqantider}).
\end{remark}

\noindent From this definition, one  easily derives a more general formula
\begin{eqnarray*}
\int f(x)\Dpq g(x)\dpq x&=& (p-q)x\sum_{k=0}^{\infty}\frac{q^{k}}{p^{k+1}}f\left(\frac{q^{k}}{p^{k+1}}x\right)\Dpq g\left(\frac{q^{k}}{p^{k+1}}x\right)\\
&=&(p-q)x\sum_{k=0}^{\infty}\frac{q^{k}}{p^{k+1}}f\left(\frac{q^{k}}{p^{k+1}}x\right)\dfrac{ g\left(\frac{q^{k}}{p^{k}}x\right)-g\left(\frac{q^{k+1}}{p^{k+1}}x\right)}{(p-q)\frac{q^{k}}{p^{k+1}}x}\\
&=&\sum_{k=0}^{\infty}f\left(\frac{q^{k}}{p^{k+1}}x\right)\left(g\left(\frac{q^{k}}{p^{k}}x\right)-g\left(\frac{q^{k+1}}{p^{k+1}}x\right)   \right),
\end{eqnarray*}
or otherwise stated
\begin{equation}
\int f(x)\dpq g(x)=\sum_{k=0}^{\infty}f\left(\frac{q^{k}}{p^{k+1}}x\right)\left(g\left(\frac{q^{k}}{p^{k}}x\right)-g\left(\frac{q^{k+1}}{p^{k+1}}x\right)   \right).
\end{equation}

\noindent We have merely derived (\ref{pqantider}) formally and have yet to examine under what conditions it really converges to a $(p,q)$-antiderivetive. The theorem below gives a sufficient condition for this.

\begin{theorem}\label{pqinttheo}
Suppose $0<\dfrac{q}{p}<1$. If $|f(x)x^\alpha|$ is bounded on the interval $(0,A]$ for some $0\leq \alpha<1$, then the $(p,q)$-integral (\ref{pqantider}) converges to a function $F(x)$ on $(0,A]$, which is a $(p,q)$-antiderivative of $f(x)$. Moreover, $F(x)$ is continuous at $x=0$ with $F(0)=0$.
\end{theorem}

\begin{proof}
Let us assume that $|f(x)x^\alpha|<M$ on $(0,A]$. For any $0<x<A$, $j\geq 0$, 
\[\left|f\left(\frac{q^j}{p^{j+1}}x\right)\right|<M\left(\dfrac{q^j}{p^{j+1}}x\right)^{-\alpha}.\]
Thus, for $0<x\leq A$, we have 
\begin{equation}\label{prof1}
\left|\dfrac{q^j}{p^{j+1}}f\left(\frac{q^j}{p^{j+1}}x\right)\right|<M\dfrac{q^j}{p^{j+1}}\left(\dfrac{q^j}{p^{j+1}}x\right)^{-\alpha}=Mp^{\alpha-1} x^{-\alpha}\left[\left(\dfrac{q}{p}\right)^{1-\alpha}\right]^j.\end{equation}
Since, $1-\alpha>0$ and $0<\dfrac{q}{p}<1$, we see that our series is bounded above by a convergent geometric series. Hence, the right-hand size of (\ref{pqantider}) converges point-wise  to some function $F(x)$. It follows directly from (\ref{pqantider}) that  $F(0)=0$. The fact that $F(x)$ is continuous at $x=0$, that is $F(x)$ tends to zero as $x\to 0$, is clear if we consider, using  (\ref{prof1})
\[\left| (p-q)x\sum_{k=0}^{\infty}\frac{q^{k}}{p^{k+1}}f\left(\frac{q^{k}}{p^{k+1}}x\right)   \right|<\dfrac{M(p-q)x^{1-\alpha}}{p^{1-\alpha}-{q}^{1-\alpha}},\quad 0<x\leq A.\]
In order to check that $F(x)$ is a $(p,q)$-antiderivative we $(p,q)$-differentiate it:
\begin{eqnarray*}
\Dpq F(x)&=& \dfrac{1}{(p-q)x}\left( (p-q)px\sum_{k=0}^{\infty}\frac{q^{k}}{p^{k+1}}f\left(\frac{q^{k}}{p^{k+1}}px\right) \right.\\
&&\hspace{3cm}\left. -(p-q)qx\sum_{k=0}^{\infty}\frac{q^{k}}{p^{k+1}}f\left(\frac{q^{k}}{p^{k+1}}qx\right)\right)\\
&=& \sum_{k=0}^{\infty}\frac{q^{k}}{p^{k}}f\left(\frac{q^{k}}{p^{k}}x\right)-\sum_{k=0}^{\infty}\frac{q^{k+1}}{p^{k+1}}f\left(\frac{q^{k+1}}{p^{k+1}}x\right)\\
&=& \sum_{k=0}^{\infty}\frac{q^{k}}{p^{k}}f\left(\frac{q^{k}}{p^{k}}x\right)-\sum_{k=1}^{\infty}\frac{q^{k}}{p^{k}}f\left(\frac{q^{k}}{p^{k}}x\right)\\
&=& f(x).
\end{eqnarray*}
Note that if $x\in(0,A]$ and $0<\dfrac{q}{p}<1$, then $\dfrac{q}{p}x\in (0,A]$, and the $(p,q)$-differentiation is valid.
\end{proof}

\begin{remark}
Note that if the assumption of (\ref{pqinttheo}) is satisfied, the $(p,q)$-integral gives the unique $(p,q)$-antiderivative that is continuous at $x=0$, up to a constant. On the other hand, if we know that $F(x)$ is a $(p,q)$-antiderivative of $f(x)$ and $F(x)$ is continuous at $x=0$, $F(x)$ must be given, up to a constant, by (\ref{pqantider}), since a partial sum of the $(p,q)$-integral is 
\begin{eqnarray*}
(p-q)x\sum_{j=0}^{N}\frac{q^j}{p^{j+1}}f\left(\frac{q^j}{p^{j+1}}x\right)&=&(p-q)x\sum_{j=0}^{N}\frac{q^j}{p^{j+1}} \Dpq\left. F(t)\right|_{t=\frac{q^j}{p^{j+1}}x}\\
&=& (p-q)x\sum_{j=0}^N\frac{q^j}{p^{j+1}}\left(\dfrac{F\left(\frac{q^j}{p^j}x\right)-F\left(\frac{q^{j+1}}{p^{j+1}}x\right)}{(p-q)\frac{q^j}{p^{j+1}}x}  \right)\\
&=& \sum_{j=0}^N\left(F\left(\frac{q^j}{p^j}x\right)-F\left(\frac{q^{j+1}}{p^{j+1}}x\right)\right)\\
&=&F(x)-F\left(\frac{q^{N+1}}{p^{N+1}}x\right)
\end{eqnarray*} 
which tends to $F(x)-F(0)$ as $N$ tends to $\infty$, by the continuity of $F(0)$ at $x=0$.
\end{remark}
\noindent Let us emphasize on an example where the $(p,q)$-derivative fails. Consider \\$f(x)=\dfrac{1}{x}$. Since
\begin{equation}
\Dpq \ln x=\dfrac{\ln px-\ln qx}{(p-q)x}=\dfrac{\ln p-\ln q}{p-q}\dfrac{1}{x},
\end{equation}
we have 
\begin{equation}
\int \dfrac{1}{x}\dpq x=\dfrac{p-q}{\ln p-\ln q}\ln x.
\end{equation}
However, the formula (\ref{pqantider}) gives 
\[\int \dfrac{1}{x}\dpq x=(p-q)\sum_{j=0}^{\infty}1=\infty.\]
The formula fails because $f(x)x^\alpha$ is not bounded for any $0\leq \alpha<1$. Note that $\ln x$ is not continuous at $x=0$.

We now apply formula (\ref{pqantider}) to define the definite $(p,q)$-integral.

\begin{definition}
Let $f$ be an arbitrary function and $a$ be a real number, we set
\begin{eqnarray}
\int_{0}^{a}f(x)\dpq x&=&
(q-p)a\sum\limits_{k=0}^{\infty}\frac{p^{k}}{q^{k+1}}f\left(\frac{p^{k}}{q^{k+1}}a\right)\quad \textrm{if }\quad \left|\dfrac{p}{q}\right|<1\label{pqint1}\\
\int_{0}^{a}f(x)\dpq x&=&(p-q)a\sum\limits_{k=0}^{\infty}\frac{q^{k}}{p^{k+1}}f\left(\frac{q^{k}}{p^{k+1}}a\right)  \quad \textrm{if }\quad \left|\dfrac{p}{q}\right|>1.\label{pqint2}
\end{eqnarray}
\end{definition}

%

\begin{remark}
Note that for $p=1$, the definition (\ref{pqint2}) reduces to the well known Jackson integral (see \cite[P.\ 67]{kac})
\[\int f(x)d_qx=(1-q)x\sum_{k=0}^{\infty}q^kf(q^kx).\]
For $p=r^{1/2}$, $q=s^{-1/2}$, 
\[\left|\frac{p}{q}\right|<1\iff |rs|<1,\] and the formula (\ref{pqint1}) reads 
\[\int_{0}^{a}f(x)\dpq x=(s^{-1/2}-r^{1/2})a\sum\limits_{k=0}^{\infty}{r^{k/2}}s^{(k+1)/2}f\left({r^{k/2}}{s^{(k+1)/2}}a\right),\]
which is the formula (11) given in \cite{burban}.
Once more, for 
For $p=r^{1/2}$, $q=s^{-1/2}$, 
\[\left|\frac{p}{q}\right|>1\iff |rs|>1,\] and the formula (\ref{pqint2}) reads 
\[\int_{0}^{a}f(x)\dpq x=(r^{1/2}-s^{-1/2})a\sum\limits_{k=0}^{\infty}{s^{-k/2}}{r^{-(k+1)/2}}f\left({s^{-k/2}}{r^{-(k+1)/2}}a\right),\]
which is the formula (10) given in \cite{burban}.
\end{remark}

\begin{definition}
Let $f$ be an arbitrary function  $a$ and $b$ be two nonnegative numbers such that $a<b$, then we set 
\begin{equation}
\int_{a}^{b}f(x)\dpq x=\int_{0}^{b}f(x)\dpq x-\int_{0}^{a}f(x)\dpq x.
\end{equation}
\end{definition}

We cannot obtain a good definition of improper integral by simply letting $a\to\infty$ in (\ref{pqint2}). Instead, since 
\begin{eqnarray*}
\int_{{q^{j+1}}/{p^{j+1}}}^{{q^{j}}/{p^{j}}}f(x)\dpq x&=&\int_{0}^{\frac{q^{j}}{p^{j}}}f(x)\dpq x-\int_{0}^{\frac{q^{j+1}}{p^{j+1}}}f(x)\dpq x\\
&=&(p-q)\left\{ \sum\limits_{k=0}^{\infty}\frac{q^{k+j}}{p^{k+1+j}}f\left(\frac{q^{k+j}}{p^{k+1+j}}\right)   -\sum\limits_{k=0}^{\infty}\frac{q^{k+j+1}}{p^{k+j+2}}f\left(\frac{q^{k+j+1}}{p^{k+j+2}}\right)\right\}\\
&=&(p-q)\dfrac{q^j}{p^{j+1}}f\left(\frac{q^{j}}{p^{j+1}}\right),
\end{eqnarray*}
it is natural to define the improper $(p,q)$-integral as follows.

\begin{definition}
The improper $(p,q)$-integral of $f(x)$ on $[0;+\infty)$ is defined to be 
\begin{eqnarray}
\int_0^{\infty}f(x)\dpq x&=&\sum_{j=-\infty}^{\infty}\int_{{q^{j+1}}/{p^{j+1}}}^{{q^{j}}/{p^{j}}}f(x)\dpq x\nonumber\\
&=& (p-q)\sum_{j=-\infty}^{\infty}\dfrac{q^j}{p^{j+1}}f\left(\frac{q^{j}}{p^{j+1}}\right)\label{improper1}
\end{eqnarray} 
if $0<\dfrac{q}{p}<1$ or 
\begin{eqnarray}
\int_0^{\infty}f(x)\dpq x&=&(q-p)\sum_{j=-\infty}^{\infty}\int_{{q^{j}}/{p^{j}}}^{{q^{j+1}}/{p^{j+1}}}f(x)\dpq x
\end{eqnarray} 
if $\dfrac{q}{p}>1$ where the formula  is used.

\end{definition}

\begin{proposition} Suppose that $0<\dfrac{q}{p}<1$.
The improper $(p,q)$-integral defined above converges if $x^\alpha f(x)$ is bounded in a neighbourhood of $x=0$ with $\alpha <1$ and for sufficiently large $x$ with some $\alpha>1$.
\end{proposition}

\begin{proof}
By (\ref{improper1}) we have 
\begin{eqnarray*}
\int_{0}^{\infty}f(x)\dpq x&=&(p-q)\sum_{j=-\infty}^{\infty}\dfrac{q^j}{p^{j+1}}f\left(\frac{q^{j}}{p^{j+1}}\right)\\
&=& (p-q)\left\{ \sum_{j=0}^{\infty}\dfrac{q^j}{p^{j+1}}f\left(\frac{q^{j}}{p^{j+1}}\right)+\sum_{j=1}^{\infty}\dfrac{q^{-j}}{p^{-j+1}}f\left(\frac{q^{-j}}{p^{-j+1}}\right)  \right\}
\end{eqnarray*}
The convergence of the first sum is proved by Theorem \ref{pqinttheo}. For the second sum, suppose for $x$ large we have $|x^\alpha f(x)|<M$ where $\alpha>1$ and $M>0$. Then, we have for sufficiently large $j$,
\begin{eqnarray*}
\left|\dfrac{q^{-j}}{p^{-j+1}}f\left(\frac{q^{-j}}{p^{-j+1}}\right)  \right|&=&p^{\alpha-1}\left(\dfrac{q}{p}\right)^{j(\alpha-1)}\left|\left(\dfrac{q^{-j}}{p^{-j+1}}\right)^{\alpha}f\left(\frac{q^{-j}}{p^{-j+1}}\right)  \right|\\
&<&Mp^{\alpha-1}\left(\dfrac{q}{p}\right)^{j(\alpha-1)}.
\end{eqnarray*}
Therefore, the second sum is also bounded above by a convergent geometric series, and thus converges.
\end{proof}
Note that similar proposition can be stated when $\dfrac{q}{p}>1$.

\begin{definition}
Let $f$ be an arbitrary function and $a$ be a nonnegative real number, then we put 
\begin{eqnarray}
\int_{a}^{\infty}f(x)\dpq x&=&
(q-p)a\sum\limits_{k=0}^{\infty}\frac{p^{-k}}{q^{-(k+1)}}f\left(\frac{p^{-k}}{q^{-(k+1)}}a\right)\quad \textrm{if }\quad \left|\dfrac{p}{q}\right|<1\label{pqint3}\\
\int_{a}^{\infty}f(x)\dpq a&=&(p-q)a\sum\limits_{k=0}^{\infty}\frac{q^{-k}}{p^{-(k+1)}}f\left(\frac{q^{-k}}{p^{-(k+1)}}a\right)  \quad \textrm{if }\quad \left|\dfrac{p}{q}\right|>1.\label{pqint4}
\end{eqnarray}
\end{definition}

\begin{remark}
Combining (\ref{pqint1}) with (\ref{pqint3}) and (\ref{pqint2}) with (\ref{pqint4}) we have for $a=1$
\begin{eqnarray}
\int_{0}^{\infty}f(x)\dpq x&=&
(q-p)\sum\limits_{k=-\infty}^{\infty}\frac{p^{k}}{q^{k+1}}f\left(\frac{p^{k}}{q^{k+1}}\right)\quad \textrm{if }\quad \left|\dfrac{p}{q}\right|<1\label{pqint5}\\
\int_{0}^{\infty}f(x)\dpq x&=&(p-q)\sum\limits_{k=-\infty}^{\infty}\frac{q^{k}}{p^{k+1}}f\left(\frac{q^{k}}{p^{k+1}}\right)  \quad \textrm{if }\quad \left|\dfrac{p}{q}\right|>1.\label{pqint6}
\end{eqnarray}
\end{remark}

\subsection{The fundamental theorem of $(p,q)$-calculus}

In ordinary calculus, a derivative is defined as the limit of a ratio, and a definite integral is defined as the limit of an infinite sum. Their subtle and surprising relation is given by the Newton-Leibniz formula, also called the fundamental theorem of calculus. Following the work done in $q$-calculus, where the  introduction of the definite integral (see \cite{kac}) has been motivated by an antiderivative, the relation between the $(p,q)$-derivative and the $(p,q)$-integral is more obvious. Similarly to the ordinary and the $q$ cases, we have the following fundamental theorem, or $(p,q)$-Newton-Leibniz formula.

\begin{theorem}(Fundamental theorem of $(p,q)$-calculus) If $F(x)$ is an antiderivative of $f(x)$ and $F(x)$ is continuous at $x=0$, we have 
\begin{equation}\label{fundemental}
\int_a^bf(x)\dpq x=F(b)-F(a),
\end{equation}
where $0\leq a<b\leq \infty$.
\end{theorem} 

\begin{proof}
Since $F(x)$ is continuous at $x=0$, $F(x)$ is given by the formula 
\[F(x)=(p-q)x\sum_{j=0}^{\infty}\dfrac{q^j}{p^{j+1}}f\left(\dfrac{q^j}{p^{j+1}}x\right)+F(0).\]
Since by definition,
\[\int_{0}^{a}f(x)\dpq x=(p-q)a\sum_{j=0}^{\infty}\dfrac{q^j}{p^{j+1}}f\left(\dfrac{q^j}{p^{j+1}}a\right),\]
we have 
\[\int_{0}^{a}f(x)\dpq x=F(a)-F(0).\]
Similarly, we have, for a finite $b$,
\[\int_{0}^{b}f(x)\dpq x=F(b)-F(0),\]
and thus 
\[\int_{a}^{b}f(x)\dpq x=\int_{0}^{b}f(x)\dpq x-\int_{0}^{a}f(x)\dpq x=F(b)-F(a).\]
Putting $a=\dfrac{q^{j+1}}{p^{j+1}}$ and $b=\dfrac{q^{j}}{p^{j}}$ and considering the definition of the improper $(p,q)$-integral (\ref{improper1}), we see that (\ref{fundemental}) is true for $b=\infty$.
\end{proof}

\begin{corollary}\label{cor1}
If $f'(x)$ exists in a neighbourhood  of $x=0$ and is continuous at $x=0$, where $f'(x)$ denotes the ordinary derivative of $f(x)$, we have 
\begin{equation}\label{fundemantal2}
\int_{a}^{b}\Dpq f(x)\dpq x=f(b)-f(a).
\end{equation}
\end{corollary}

\begin{proof}
Using L'Hospital's rule, we get 
\begin{eqnarray*}
\lim_{x\to 0}\Dpq f(x)&=&\lim_{x\to 0}\frac{f(px)-f(qx)}{(p-q)x}\\
&=& \lim_{x\to 0}\frac{pf'(px)-qf'(qx)}{p-q}=f'(0).
\end{eqnarray*}
Hence $\Dpq f(x)$ can be made continuous at $x=0$ if we define\linebreak $(\Dpq f)(0)=f'(0)$, and (\ref{fundemantal2}) follows from the theorem.
\end{proof}
As the $q$-integral, an important difference between the $(p,q)$-integral an the its ordinary counterpart is that even if we are integrating a function on an interval like $[1;2]$, we have to care about behaviour at $x=0$. This has to do with the definition of the definite $(p,q)$-integral and the condition for the convergence of the $(p,q)$-integral.

Now suppose that $f(x)$ and $g(x)$ are two functions whose ordinary derivatives exists in a neighbourhood of $x=0$. Using the product rule (\ref{productrule2}), we have
\[\Dpq (f(x)g(x))= f(px)\Dpq g(x)+g(qx)\Dpq f(x),\label{prorule2}.\]
Since the product of differentiable functions is also differentiable in ordinary calculus, we can apply (\ref{cor1}) to obtain 
\[f(b)g(b)-f(a)g(a)=\int_a^b f(px)\left(\Dpq g(x)\right)\dpq x+\int_a^b g(qx)\left(\Dpq f(x)\right)\dpq x,\]
or 
\[\int_a^b f(px)\left(\Dpq g(x)\right)\dpq x=f(b)g(b)-f(a)g(a)-\int_a^b g(qx)\left(\Dpq f(x)\right)\dpq x,\]
which is the formula of $(p,q)$-integration by part. Note that $b=\infty$ is allowed.

%
%



\vspace*{3cm}
\flushleft{ {\bf P. Njionou Sadjang}\\ University of Yaounde 1\\ Advanced School of Education \\ P.O Box 47, Yaounde\\ \texttt{pnjionou@yahoo.fr}\\ Ph. +237 99 61 34 40}

\label{lastpage}


\begin{thebibliography}{12}
\markboth{Taylor \& Francis and I.T. Consultant}{Integral Transforms and Special Functions}

\bibitem{desire} J.D. Bukweli-Kyemba, M.N. Hounkonnou, {\em Quantum deformed algebras : coherent states and special functions}, arXiv:1301.0116v1, 2013.

\bibitem{burban}
I. M. Burban, A. U. Klimyk (1994),  {\em $(P,Q)$-Differentiation, $(P,Q)$ integration, and $(P,Q)$-hypergeometric functions related to quantum groups}, Integral Transform and Special Functions, 2:1, 15-36.

\bibitem{Jagannathan1998} R. Jagannathan, {\em $(P,Q)$-Special Functions}, arXiv:math/9803142v1, 1998.


\bibitem{Jagannathan2006}
R. Jagannathan, K. Srinivasa Rao, {\em Tow-parameter quantum algebras, twin-basic numbers, and associated generalized hypergeometric series}, arXiv:math/0602613v, 2006.

\bibitem{kac} V. Kac, P. Cheung: \emph{Quantum calculus}, Springer, 2001.


\end{thebibliography}
\end{document}